\documentclass[letterpaper,10pt]{article}

\usepackage{graphicx}
\usepackage{palatino}
\usepackage{subfig}
\usepackage{amssymb}
\usepackage{amsthm}
\usepackage{amsmath}
\usepackage{epstopdf}
\usepackage{cite}
\usepackage{paralist}
\usepackage{graphics}
\usepackage{epsfig}
\usepackage{color}
\usepackage{setspace}
\usepackage[affil-it]{authblk}

\newtheorem{theorem}{Theorem}[section]

\newtheorem{lemma}[theorem]{Lemma}
\newtheorem{proposition}{Proposition}

\newcommand{\ep}{\varepsilon}
\newcommand{\eps}[1]{{#1}_{\varepsilon}}

\title{Variable Step Size Multiscale Methods for Stiff and Highly Oscillatory Dynamical Systems}

\author{Yoonsang Lee%
	\thanks{Corresponding author. \texttt{ylee@math.utexas.edu}  }}
	\affil{Department of Mathematics, The University of Texas at Austin, Austin, Texas 78712}
	
\author{Bjorn Engquist%
	\thanks{\texttt{engquist@math.utexas.edu}  }}
	\affil{Department of Mathematics and ICES, The University of Texas at Austin, Austin, Texas 78712}
\begin{document}
\maketitle

\bigskip


\begin{abstract}
We present a new numerical multiscale integrator for stiff and highly oscillatory dynamical systems. The new algorithm can be seen as an improved version of the seamless Heterogeneous Multiscale Method by E, Ren, and Vanden-Eijnden and the method FLAVORS by Tao, Owhadi, and Marsden. It approximates slowly changing quantities in the solution with higher accuracy than these other methods while maintaining the same computational complexity. To achieve higher accuracy, it uses variable  mesoscopic time steps which are determined by a special function satisfying moment and regularity conditions. Detailed  analytical and numerical comparison between the different methods are given.
\end{abstract}

\section{Introduction}\label{INTRO}
We consider numerical solutions of stiff and highly oscillatory dynamical systems of the form
\begin{equation}\label{eq:model}\frac{dx}{dt}=\eps{f}(x),\quad x(0)=x_0\end{equation}
where the Jacobian of $f_{\epsilon}$ has eigenvalues with large negative real parts or purely imaginary eigenvalues of large modulus. That is, the spectral radius is of the order,
$$\rho\left(\frac{\partial \eps{f}}{\partial x}\right)=\mathcal{O}\left(\frac{1}{\epsilon}\right)\gg 1,\quad 0<\epsilon\ll 1.$$
This imposes severe restrictions on the time steps.
Resolving the $\epsilon$ scale requires the time steps of a traditional direct numerical simulation (DNS) to be of order $\mathcal{O}(\epsilon)$ or less.

There are many numerical methods to approximate the solutions of (\ref{eq:model})  with less computational complexity than $\mathcal{O}(\frac{1}{\epsilon})$ for $\mathcal{O}(1)$ time intervals. Exponential integrators or Gautschi type methods \cite{ODE3,EXPINT1,GAUTSCHI} use an analytic form for the most significant part of the oscillatory solutions resulting in significantly less restriction on the time steps from stability and accuracy. Another method for highly oscillatory problems using asymptotic expansions in inverse powers of the oscillatory parameter (\cite{ISERLES} and references therein) has computational cost essentially independent of the oscillatory parameter.

In this paper, we focus on the following two forms of the model problem which have scale separation. First, we consider the problem with explicitly identified slow and fast variables,
\begin{equation}\label{eq:model1}
\begin{split}
\frac{d\xi}{dt}=&f_0(\xi,\eta ),\quad \xi(0)=\xi_0\\
\frac{d\eta}{dt}=&\frac{f_1(\xi,\eta )}{\epsilon},\quad\eta(0)=\eta_0,\quad 0<\epsilon \ll 1
\end{split}
\end{equation}
where $\eta$ is ergodic on some invariant manifold $\mathcal{M}(\xi)$ for fixed $\xi$. We also consider another problem 
\begin{equation}\label{eq:model2}
\frac{dx}{dt}=f_0(x)+\frac{f_1(x)}{\epsilon},\quad x(0)=x_0,\quad 0<\epsilon\ll 1
\end{equation}
where the unperturbed equation
$$\frac{dy}{dt}=\frac{f_1(y)}{\epsilon},\quad y(0)=y_0$$
is ergodic on some invariant manifold $\mathcal{M}(y_0)$.

In (\ref{eq:model1}), $\eta$ is called the fast variable because it has fast transient or highly oscillatory behavior when the Jacobian of $f_1$ has negative real parts or all imaginary parts. The slow variable $\xi(t)$ can be consistently approximated in any $\mathcal{O}(1)$ time by an averaged equation
$$\frac{d\Xi}{dt}=\bar{f}(\Xi):=\int_{\mathcal{M}(\Xi)}f_0(\Xi,\eta)d\mu(\Xi,\eta)$$
where $\mu(\Xi,\eta)$ is the invariant measure of $\eta$ for fixed $\Xi$. For more details, see \cite{AVG1} and \cite{AVG2}. In the case of (\ref{eq:model2}), it is often assumed for the analysis that there exists a diffeomorphism from $x$ to $(\xi,\eta)$ and this implies that there exist hidden slow variables in (\ref{eq:model2}). The existence of slow variables for these problems motivate the development of efficient numerical schemes for integrating the slow components of slow-fast systems without resolving all fast variables.

\begin{figure}
	\begin{center}
	\subfloat[HMM]{\label{fig:HMM}\includegraphics[width=7.3cm]{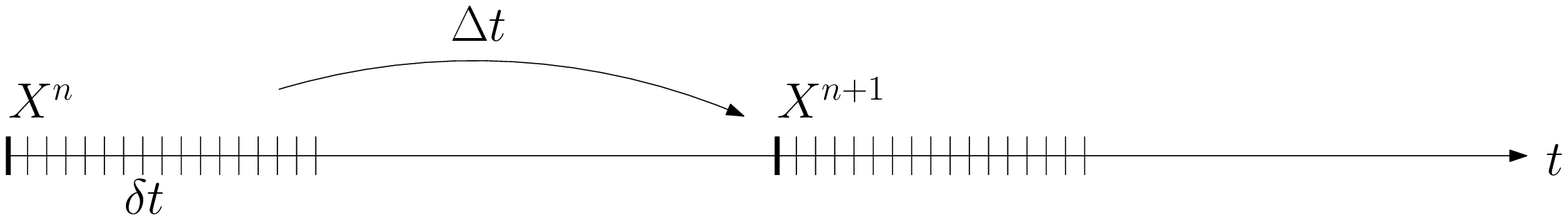}}\quad
	\subfloat[MSHMM]{\label{fig:MSHMM}\includegraphics[width=7.3cm]{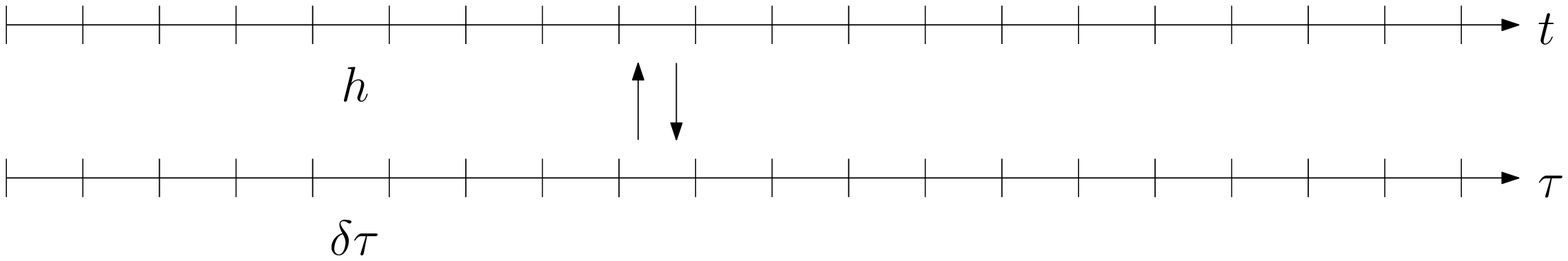}}\\
	\subfloat[FLAVORS]{\label{fig:FLAVORS}\includegraphics[width=7.3cm]{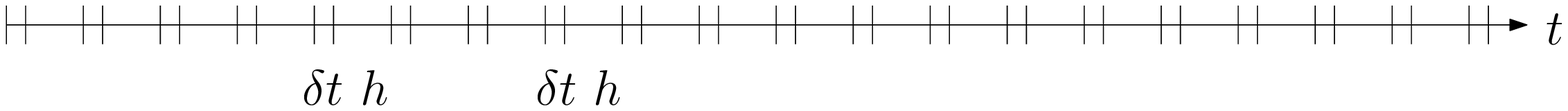}}\quad
	\subfloat[VSHMM]{\label{fig:NEW}\includegraphics[width=7.3cm]{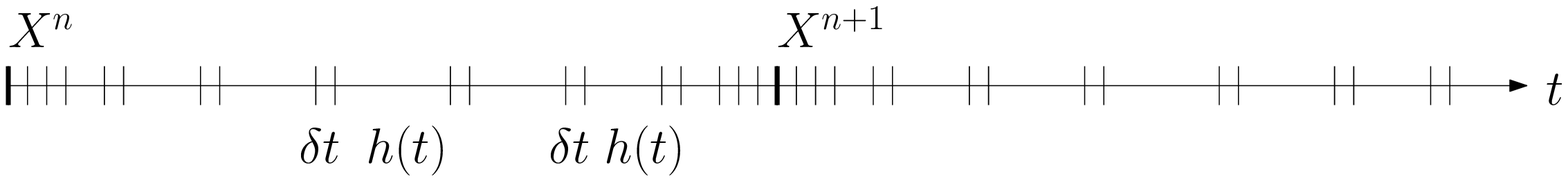}}\\
	\caption{Schematics of HMM, MSHMM, FLAVORS and VSHMM}\label{fig:schematics}
	\end{center}
\end{figure}

In this paper, we focus on Heterogeneous Multiscale Methods (HMM) framework \cite{generalHMM,HMM} that captures the effective behavior of the slow variables on the fly by solving the full scale problem in very short time intervals (see Figure \ref{fig:HMM}). HMM does not require any a priori information about the effective force, and has a suitable filtering kernel to estimate the force by time averaging of the local microscopic solution. Because the effective force is independent of $\epsilon$ and the fast variables, the time step for the slow variables which is called macro time step in HMM can be chosen independently of $\epsilon$.

One of the variants of HMM, the seamless Heterogeneous Multiscale Method first introduced by Fatkullin, and Vanden-Eijnden in \cite{FATKULLIN} and further developed by E, Ren, and Vanden-Eijnden in \cite{MSHMM} modifies HMM in that it does not require reinitialization of microscale simulation at each macro time step or each macro iteration step. In this strategy, the macro- and micro-models in (\ref{eq:model1}) use different time steps and exchange data at every step. The macroscale solver uses a mesoscale time step that is much finer than the one in HMM for the effective system, in order for the microscale system to relax and influence the macro scale (see Figure \ref{fig:MSHMM}). We will here label the method MSHMM for mesoscale HMM in order to differentiate it from other methods called seamless HMM (SHMM) for multi spatial and multi time scales without scale separation \cite{SUPERPARAMETRIZATION,TURBULENCE}.

A similar technique can also be used to solve the system of the form (\ref{eq:model2}) without identification of the slow and fast variables beforehand. It was first noted by Vanden-Eijnden \cite{ERIC}  and  later a variant was proposed and further developed by Tao, Owhadi and Marsden \cite{FLAVORS} called 'Flow Averaging Integrators' (FLAVORS). It is based on the averaging of the instantaneous flow of the system with hidden slow and fast variables instead of capturing the effective force of the slow variables. By turning the stiff part on over a microscopic time step ${\delta t}$ and off during a mesoscopic time step $h$, FLAVORS obtains computational efficiency (see Figure \ref{fig:FLAVORS}).

In Section \ref{MSHMMFLAVORS}, we show that MSHMM and FLAVORS share a common characteristic in that they both approximate the effective behavior of (\ref{eq:model1}) and (\ref{eq:model2}) respectively by solving the problem with increased $\epsilon$ values.

The increase of the $\epsilon$ value gives computational efficiency better than a direct approximation of the original problem but at the cost of reduced accuracy. The amplitude in highly oscillatory solutions, for example, is increased which is related to the increased $\epsilon$ value. Because of this loss of accuracy, it is difficult to generate higher order approximation of the effective behavior in both methods.

The goal of the proposed method that applies to the more general formulation (\ref{eq:model2}) is to increase accuracy by controlling the transient and the amplified oscillations while keeping the same computational efficiency and structure of implementation in the methods discussed above. And we call our proposed method a variable step size Heterogeneous Multiscale Method (VSHMM). To gain the control of the oscillations, the method uses variable mesoscopic time steps which are determined by a special function with a certain moment condition and regularity properties which is described in Section \ref{VSHMM}. Given a macro time step at which we want to sample the value of the averaged solution, the mesoscopic time step increases smoothly from fine one to a coarser mesoscale time step to obtain efficiency in computation. Once it reaches close to the next macro time step, the time step decreases again back to the original size and repeat this process for the next macro time step (see Figure \ref{fig:NEW} ).

VSHMM can also be used for many well-separated scale problems without using hierarchical iteration. Hierarchical iteration using the other multiscale methods - HMM, VSHMM and FLAVORS - has computational complexity which increases exponentially as the number of different scales increases \cite{threescale}. Using the variable step size method, we can develop a new method whose complexity increases proportional to the number of different scales. Here, we focus on the two well separated scale problems and the new method for many scales will be reported in a forthcoming paper by the authors \cite{VSSHMM}. The basic idea is to include different components of the force depending on the variable step size, from the full $f_{\epsilon}(x)$ for the shortest step size to only the slowest components for the longest step size. The intermediate step size will contain the intermediate to slow components of $f_{\epsilon}(x)$.

As stated above, we have to mention that VSHMM requires scale separation and ergodicity of the fast variables. For stiff dissipative problems without scale separation, efficient methods exist such as implicit methods for small systems and Chebyshev methods for large systems. Here we study the application of VSHMM to dissipative problems with a potential application of VSHMM for concurrent multiscale problems in mind.

This paper is organized in the following way. In Section \ref{MSHMMFLAVORS}, we review MSHMM and FLAVORS and show that they are equivalent in that they both solve a modified equation with increased $\epsilon$ value. In Section \ref{VSHMM}, we propose a new method as an extension of MSHMM or FLAVORS to control the transient and the amplified errors and introduce higher order methods. In Section \ref{ANAL}, we analyze the proposed method for dissipative and highly-oscillatory systems. In Section \ref{NUME}, numerical examples of dissipative and highly oscillatory systems are shown and also higher order method is verified.

\section{MSHMM and FLAVORS}\label{MSHMMFLAVORS}
In this section, we review and compare MSHMM and FLAVORS. They are shown to share common characteristics except in modifications to the way time-stepping is implemented.

The philosophy behind MSHMM \cite{MSHMM} is that we use different clocks for slow and fast variables. It requires identification of slow and fast variables in advance and is applicable to (\ref{eq:model1}).
If we denote the micro and mesoscopic time steps by $\delta \tau$ and $h$, an explicit first order MSHMM solves for $\eta$ first,
$$\eta^{n+1}=\eta^n+\frac{{\delta \tau}}{\epsilon}f_1(\xi^n,\eta^n)=\eta^n+\frac{h}{\epsilon'}f_1(\xi^n,\eta^n)$$
with $\epsilon'=\epsilon\frac{h}{{\delta \tau}}$ and it uses the information from this calculation for the evolution of $\xi,$
$$\xi^{n+1}=\xi^n+h f_0(\xi^n,\eta^{n+1}).$$
This is a consistent approximation to the model problem with $\epsilon$ modified to the increased value $\epsilon'=\epsilon\frac{h}{{\delta \tau}}>\epsilon$.

In \cite{FLAVORS}, Tao\textit{ et al.} propose another method based on the averaging of the instantaneous flow of the system, which is called FLAVORS. It turns on and off the stiff parts to capture the effective flow of the slow variables. It solves the full problem (\ref{eq:model2}) with the stiff part with a micro step ${\delta t}$,
$$x^{n+*}=x^n+{\delta t}\left(f_0(x^n)+\frac{f_1(x^n)}{\epsilon}\right).$$
It then uses a mesoscopic time step $h$ without the stiff part
$$x^{n+1}=x^{n+*}+hf_0(x^{n+*}).$$
If we consider this as one single step, it is 
$$x^{n+1}=x^n+\left({\delta t}f_0(x^n)+hf_0\left(x^n+{\delta t}\left(f_0(x^n)+\frac{f_1(x^n)}{\epsilon}\right)\right)\right)+\frac{({\delta t}+h)}{\epsilon'}f_1(x^n)$$
$$=x^n+({\delta t}+h)\left(f_0(x^{n+**})+\frac{f_1(x^n)}{\epsilon'}\right)$$
with an $\epsilon$ value increased to $\epsilon'=\frac{{\delta t}+h}{{\delta t}}\epsilon$ and $x^{n+**}$ such that 
$$f_0(x^{n+**})=\frac{1}{{\delta t}+h}\left({\delta t}f_0(x^n)+hf_0(x^{n+*})\right).$$
Hence, it solves the model problem with modification to $\epsilon$ and a minor difference in the $f_0$ term.

If we apply FLAVORS to (\ref{eq:model1}), it becomes much clearer that MSHMM and FLAVORS share common characteristics. As before, we use the explicit first order Euler method for each step for $\xi$ and $\eta$.  We have
$$\eta^{n+1}=\eta^n+\frac{{\delta t}}{\epsilon}f_1(\xi^n,\eta^n)$$
and
$$\xi^{n+*}=\xi^n+{\delta t}f_0(\xi^n,\eta^n)$$
$$\xi^{n+1}=\xi^{n+*}+hf_0(\xi^{n+*},\eta^{n+1})$$

The evolution of $\xi$ can be represented in a compact form,
\begin{equation}
\begin{split}
\xi^{n+1}=& \xi^n+{\delta t}f_0(\xi^n,\eta^n)+hf_0(\xi^{n+*},\eta^{n+1})\\
=&\xi^n+({\delta t}+h)\left(\theta f_0(\xi^n,\eta^n)+(1-\theta)f_0(\xi^{n+*},\eta^{n+1})\right)\\
=&\xi^n+({\delta t}+h)\left(\theta f_0(\xi^n,\eta^n)+(1-\theta)f_0(\xi^{n},\eta^{n+1})\right)+\mathcal{O}({\delta t}h)
\end{split}
\end{equation}
where $\theta=\frac{{\delta t}}{{\delta t}+h}$.

If we now choose
$${\delta \tau} \textrm{ in MSHMM}={\delta t} \textrm{ in FLAVORS}$$
and
$$h \textrm{ in MSHMM}={\delta t}+h\textrm{ in FLAVORS},$$
then they both are first order approximations (with slightly different fact that FLAVORS uses the $\theta$ method for the slow variable) to the following modified equation, with increased $\epsilon'=\frac{{\delta t}+h}{{\delta t}}\epsilon=(1+\alpha)\epsilon,$ where $\alpha=\frac{h}{{\delta t}}$,
\begin{equation}\label{eq:mod}
\begin{split}
\frac{d}{dt}{\tilde{\xi}}=&f_0(\tilde{\xi},\tilde{\eta} ),\qquad 0\leq t\leq T\\
\frac{d}{dt}{\tilde{\eta}}=&\frac{1}{(1+\alpha)\epsilon}f_1(\tilde{\xi},\tilde{\eta} )
\end{split}
\end{equation}

For simplicity, the reduction in the overall processes of time integration, which is $\alpha$, is called the savings factor. The savings factor gives information for computational efficiency of MSHMM or FLAVORS but it is not the actual computational efficiency of the methods. For FLAVORS, as an example, if the same order integrators are used for the forces with and without the stiff part, the number of function evaluations of the force terms for direct numerical simulations (DNS) and FLAVORS using the same micro time step $\delta t$ are given by $\lceil 1+\alpha\rceil$ and $2$ for the time $\delta t+h$ where $\lceil\cdot\rceil$ is the ceiling or the smallest integer function. Therefore, the computational efficiency of FLAVORS over DNS is 
\begin{equation}\label{eq:efficiency}
\frac{\lceil 1+\alpha\rceil}{2}.
\end{equation}

For better computational efficiency, it is obvious to use larger $\alpha$ values. But arbitrarily large $\alpha$ values do not guarantee the convergence of the methods to slow variables. In \cite{FLAVORS}, the relation between $\delta t$, $h$ and $\epsilon$ is analyzed for convergence:
\begin{equation}\label{eq:stepcond}
\frac{\delta t^2}{\epsilon^2}\ll h+\delta t\ll \frac{\delta t}{\epsilon}.
\end{equation}
Using $\alpha$ instead of $h$, we can check
\begin{equation}\label{eq:alphacond}
(\alpha+1) \epsilon\ll 1
\end{equation}
for the convergence of FLAVORS.

To see the effect of the increased $\epsilon'=(1+\alpha)\epsilon$ value, we compare the solution $\xi(t)$ of (\ref{eq:model1}) with the solution $\tilde{\xi}(t)$ of (\ref{eq:mod}) for $\mathcal{O}(1)$ time $t$. Let $\Xi(t)$ be the effective solution of (\ref{eq:model1}). Because (\ref{eq:mod}) has the same invariant measure of (\ref{eq:model1}), $\Xi(t)$ is also the effective solution of (\ref{eq:mod}). For the averaging error, it is normally expected to have the following type of error bound (see \cite{generalHMM,HMM,AVG1} for example)
\begin{equation}\label{eq:avgerr}\|\xi(t)-\Xi(t)\|_{\infty}\leq C\epsilon^{a}\end{equation}
for constants $a>0$ and $C$ which is dependent on time $t$ and independent of $\epsilon$. Similarly, the averaging error of $\tilde{\xi}$ by $\Xi$ is given by
$$\|\tilde{\xi}(t)-\Xi(t)\|_{\infty}\leq C(1+\alpha)^a\epsilon^{a}$$
which implies the following amplified error due to the increased $\epsilon'$ value.
$$\|\xi(t)-\tilde{\xi}(t)\|_{\infty}\leq C(1+\alpha)^a\epsilon^{a}.$$
If we want more computational efficiency, the savings factor $\alpha=\frac{h}{{\delta t}}$ should be larger. However then we lose accuracy because of the amplified oscillations. 
The key feature of the proposed method is to decrease $\mathcal{O}\left((\alpha\epsilon)^a\right)$ term to $\mathcal{O}(\epsilon^a)$ which is independent of the savings factor, $\alpha$, while keeping the same computational efficiency depending on $\alpha$. Because we are looking for effective behavior of systems with $\mathcal{O}(\epsilon)$ perturbations, the diminished oscillation and fluctuation to $\mathcal{O}(\epsilon)$ is accurate enough to approximate the effective behavior of the slow variables.

\section{A Variable Step Size Mesoscale HMM (VSHMM)}\label{VSHMM}
In this section, we propose a variable step size Heterogeneous Multiscale Method (VSHMM) which controls the transient and the amplified oscillations of MSHMM and FLAVORS while maintaining the computational complexity and general structure of the methods. The new method is a modification of MSHMM and FLAVORS. The key feature of the new method is to use variable sizes of mesoscopic time step.

In dissipative problems, when all eigenvalues of the Jacobian of $f_1$ or the real part of them are negative, transient behavior of the fast variable to the quasi stationary state contributes a significant part of the error \cite{ODE2}. If $\epsilon$ is modified to a greater value, then the fast variables change slower, resulting in an error that remains in the quasi stationary solution after the transient. Therefore, it is necessary to use small $\epsilon$ values at the beginning of each macroscopic time step to guarantee that the fast variables relaxed rapidly.

In highly oscillatory problems, when the eigenvalues of the Jacobian of $f_1$ are imaginary,  as we mentioned in the previous section, the increased $\epsilon$ amplifies the oscillations and this dominates the error which can be controlled using higher order methods. If there is no a priori identification of slow variables, the only possible way to control this amplified error is to use a smaller $\epsilon$ value locally which requires a finer time step. 

Our method, VSHMM, reconcile these contradictory situations by introducing time dependent $\epsilon$ values. At the beginning and the end of each macro time step, it uses very fine mesoscopic time steps to overcome problems such as the delayed relaxation of the fast variables in dissipative problems and amplified oscillations in highly oscillatory problems,  while using coarse mesoscopic time steps at all other times to save computational complexity (see Figure \ref{fig:schematics} for comparison of time stepping with the other method). Once the system has evolved to the next macro time step,  we iterate the same process.

Therefore, by using the variable mesoscopic step sizes, we expect to obtain a more accurate approximation of the slow variables than MSHMM and FLAVORS, only after macro time steps. We emphasize that there is no explicit macro time stepping in the new method but we use the calculated values only at the specified macro time intervals because other values are not guaranteed to give less amplified errors. In the highly oscillatory problems, if we want to have the same savings factor as MSHMM and FLAVORS, the intermediate values of the new method between two macroscopic sampling time becomes more oscillatory than MSHMM or FLAVORS (see Section \ref{NUME} for numerical examples). This is because $\epsilon$ is modified for a larger value than the modified $\epsilon$ of MSHMM or FLAVORS to compensate the loss of efficiency at the beginning and the end of each macro time step.

\subsection{Description of the new method}
In the proposed method, the mesoscopic time step $h$ for MSHMM or FLAVORS is described as a time dependent function. Let $K\in C^q_c((0,1))$ with a compact support in $(0,1)$ such that
\begin{eqnarray}
	\int_0^1K(t)dt&=&1,\\
	\frac{d^rK(t)}{dt^r}&=&0,\quad r=0,1,...,q \textrm{ for }t=0,1.
\end{eqnarray}
For a given macro time step $\Delta T$, the time dependent mesoscopic time step $h(t)$ is given by
\begin{equation}\label{eq:mesostep}h(t)=\alpha {\delta t}{K}_{\Delta T,q}(\Theta_{\Delta T,q}^{-1}(t\mod \Delta T)),\end{equation}
for a savings factor $\alpha>1$ when $K_{\Delta T,q}$ is a rescaled version of $K$,
$$K_{\Delta T,q}=\frac{1}{\Delta T}K(\frac{t}{\Delta T})$$
and $\Theta_{\Delta T,q}(t)$ is the antiderivative of ${K}_{\Delta T,q}$ with $\Theta_{\Delta T,q}(0)=0$.

It can be easily verified that $K_{\Delta T,q}$ satisfies the following moment and regularity conditions
\begin{eqnarray}
	\label{eq:kerprop1}\int {K}_{\Delta T,q}dt&=&\Delta T,\\
	\label{eq:kerprop2}\frac{d^r {K}_{\Delta T,q}(t)}{dt^r}&=&0,\quad r=0,1,...,q,\quad t=0,\Delta T.
\end{eqnarray}

Because (\ref{eq:model1}) can be seen a special case of (\ref{eq:model2}), we describe the proposed method for the case of (\ref{eq:model2}).

\bigskip
\textbf{ALGORITHM - one macro time step integration of VSHMM}\\
Let $\tilde{x}^n$ be the numerical solution to (\ref{eq:model2}) at $t=t^n:=n\Delta T$ with a savings factor $\alpha$.
\begin{enumerate}
	\item Integrate the full system for $\delta t$ to resolve the fast time scale
	$$\hat{x}(t^n+{\delta t})=\Phi^{\epsilon}_{\delta t} \tilde{x}(t^n)$$
	where $\Phi^{\epsilon}_{\delta t}$ is an integrator of $dx/dt=f_0(x)+f_1(x)/\epsilon$ for $\delta t$.
	\item Update time 
	$$t=t^n+{\delta t}.$$
	
	\item Integrate the system without the stiff part with mesoscopic time step $h(t)$
	$$\hat{x}(t+h(t))=\Phi^0_{h(t)} \hat{x}(t)$$
	where $\Phi^0_{h(t)}$ is an integrator of $dx/dt=f_0(x)$ for $h(t)$.
	\item Update time
	$$t=t+h(t).$$
	\item If time is at the macroscopic time points, sample the solution
	$$\tilde{x}^{n+1}=\hat{x}(t)\quad \textrm{ if } t=(n+1)\Delta T, n\in\mathbb{N}.$$
	\item Repeat from 1 for the next macro time step integration.

\end{enumerate}
\bigskip
		
The mesoscopic step size $h(t)$ is very small when the simulation starts and smoothly increases. Once it reaches $\frac{\Delta T}{2}$, it decreases smoothly back to the small value again as $t\to \Delta T$. Equivalently, the new method solves the system with small $\epsilon$ value and the $\epsilon$ value increases smoothly to accelerate the computation and returns to the original small value for the next coarse step.

In VSHMM, the ratio between the mesoscopic and microscopic time steps is not constant. The following proposition illustrates that for a given savings factor $\alpha$ in (\ref{eq:mesostep}), the computational efficiency of VSHMM is same as the case when the mesoscopic time step is constant with the same savings factor.


\begin{proposition}\label{prop}
$$\alpha=\frac{1}{\Delta T}\int_{0}^{\Delta T}\frac{h(t)}{{\delta t}}dt.$$
\end{proposition}
\begin{proof}
This is a simple application of change of variables. Let $s=\Theta_{\Delta T,q}(t)$, then
$$\frac{1}{\Delta T}\int_0^{\Delta T}\frac{h(t)}{{\delta t}}dt=\frac{1}{\Delta T}\int_0^{\Delta T}\alpha{K}_{\Delta T,q}(t)\frac{ds}{{K}_{\Delta T,q}(t)}=\alpha$$
from the moment condition of $K_{\Delta T,q}$.
\end{proof}

\subsection{Higher Order Methods}
If there is a priori identification of the slow and fast variables, MSHMM may be implemented with higher order methods and this can also be applied to VSHMM. Without identification of slow and fast variables,  it is not easy to implement a higher order scheme for the slow variables. We show that for VSHMM a second order mesoscopic integrator, for example, an explicit Runge-Kutta method, gives quadratic decrease of errors with an additional error term which can be ignored in comparison with the dominating error.

Here we describe a second order approximation. We integrate the full system with higher order numerical method for $\delta t$ to resolve the fast time scale,
\begin{equation}
\nonumber \hat{x}(t)=\Phi^{\epsilon}_{\delta t}(\hat{x}^0).
\end{equation}
Then we use the second order explicit Runge-Kutta method for $f_0(\cdot)$ part without $f_1$ part in (\ref{eq:model2}),

\begin{eqnarray}
\nonumber \hat{x}^{*}&=&\hat{x}(t)+\frac{h(t)}{2}f_0(\hat{x}(t))\\
\nonumber \hat{x}(t+h(t))&=&\hat{x}(t)+h(t)f_0(\hat{x}^{*})
\end{eqnarray}

In MSHMM and FLAVORS, the amplified error dominates other error terms from mesoscopic and microscopic integrators. Because MSHMM and FLAVORS cannot control this amplified error, it is difficult to see the effect of the higher order mesoscopic integrators. With VSHMM which controls the amplified error, higher order mesoscopic integrators can be verified (see Figure \ref{fig:order} in Section \ref{NUME} for a numerical result).

\section{Analysis}\label{ANAL}

We analyze VSHMM for highly oscillatory problems. First, we start with a review of the dissipative case and address the importance of the rapid relaxation of the fast variables at the beginning of the simulation. 
The following result for the dissipative problem is from \cite{STUART}.
\begin{theorem}\cite{STUART}
Assume that for fixed $\xi$ of (\ref{eq:model1}), $\eta$ has a unique exponentially attracting fixed point, uniformly in $\xi$. Specifically we assume that there exists $\rho$ and $a>0$ such that, for all $\xi$ and all $\eta_1,\eta_2$,
$$f_1(\xi,\rho(\xi))=0,$$
$$\langle f_1(\xi,\eta _1)-f_1(\xi,\eta _2),\eta_1-\eta_2\rangle \leq -a|\eta_1-\eta_2|^2.$$
Also assume that there exists a constant $C>0$ such that
$$|f_0(\xi,\eta )|\leq C,\quad |\nabla_x f_0(\xi,\eta )|\leq C$$
$$|\nabla_y f_0(\xi,\eta )|\leq C,\quad |\eta(x)|\leq C$$
and 
$$|\nabla \eta(x)|\leq C.$$
If $\Xi(t)$ is the solution to
$$\frac{d}{dt}{\Xi}=\bar{f}(\Xi,\rho(\Xi)),\quad \Xi(0)=\xi(0),$$

Then there are constants $M,c>0$ such that
$$|\xi(t)-\Xi(t)|^2\leq ce^{Mt}(\epsilon|\eta(0)-\rho(\xi(0))|^2+\epsilon^2)$$
\end{theorem}

This theorem indicates that the $\mathcal{O}(1)$ error in $\eta$ for the first time step may give an $\mathcal{O}(\epsilon)$ global error. If the $\epsilon$ value does not change at the beginning of simulation to guarantee that $\eta$ relaxes close enough to $\rho(\xi(0))$ and then increases to a larger value after relaxation, for example $\mathcal{O}(\sqrt{\epsilon})$, to get computational efficiency, the global error is still $\mathcal{O}(\epsilon)$. Therefore, it is important to have a small $\epsilon$ value at the beginning of simulation at each macro time step in VSHMM (see Figure \ref{fig:dissipative} in Section \ref{NUME} for a numerical example).

\subsection{Highly Oscillatory Case}
In the highly oscillatory case, the effect of variable mesoscopic time integration is analyzed and we show that the error is of order $\epsilon$ independent of $\alpha$, which is significantly less than MSHMM and VLAFORS for large $\alpha$ values.

Instead of regarding the new method as solving with an increased $\epsilon$, we can rescale time resulting in multiplication by a factor in the equation for the slow variables. Using the same procedure in Section \ref{MSHMMFLAVORS}, it can be verified that the explicit Euler version of the proposed method for (\ref{eq:model1}) for $0<t<T$ with (\ref{eq:mesostep}) is equivalent to solving the following modified problem,
\begin{equation}
\begin{split}\label{eq:model3}
\frac{d}{dt}{\xi}=&(1+\alpha {K}_{\Delta T,q}((1+\alpha)t))f_0(\xi,\eta ),\qquad 0\leq t\leq \frac{T}{1+\alpha}\\
\frac{d}{dt}{\eta}=&\frac{1}{\epsilon}f_1(\xi,\eta )
\end{split}
\end{equation}
where ${K}_{\Delta T,q}(t)$ satisfies (\ref{eq:kerprop1}) and (\ref{eq:kerprop2}). Note that in the formulation above, we do not have $\Theta^{-1}(t)$ as an argument of ${K}_{\Delta T,q}$ while the mesoscopic step sizes are given by (\ref{eq:mesostep}) which is
$$h(t)=\alpha{K}_{\Delta T,q}(\Theta^{-1}(t)).$$

In many oscillatory situations, $\eps{f}$ assumes special forms such as $\eps{f}(t)=\eps{f}(t,t/\epsilon)$ which are periodic in the second variable. We hypothesize that the effective force of the system can be defined by
$$\bar{f}(t)=\lim_{\delta\to 0}\left[\lim_{\epsilon\to 0}\frac{1}{\delta}\int_t^{t+\delta}\eps{f}(\tau)d\tau\right]$$
as in \cite{generalHMM} and \cite{HMM}.

Based on this hypothesis, the averaging error (\ref{eq:avgerr}) has $a=1$ in highly oscillatory problems which is $\mathcal{O}(\alpha\epsilon)$. The next theorem shows the effect of the rescaled system using the time dependent mesoscopic rescaling function. The rescaled system approximates the averaged system with an $\mathcal{O}(\epsilon)$ error term which is independent of $\epsilon$.

\begin{theorem} Let $(\xi,\eta)$ and $(\tilde{\xi},\tilde{\eta})$ be the solutions to (\ref{eq:model1}) and (\ref{eq:model3}) respectively with a savings factor $\alpha$ and $K_{\Delta T,q}$ satisfying (\ref{eq:kerprop1}) and (\ref{eq:kerprop2}) and the fast variables are periodic. Further assume that $\nabla_x f_0(x,y)$ is bounded independently of $\epsilon$. Then 
$$\|\xi(\Delta T)-\tilde{\xi}(\Delta T)\|_{\infty}\leq C_1\ep
+C_2\sum_{r=1}^q\frac{\alpha^r\epsilon^{r+1}}{{\Delta T}^{r-1}}
+C_3\frac{(\alpha\epsilon)^{q+1}}{{\Delta T}^q}.$$
where $C_i,i=1,2,3$, are constants independent of $\epsilon$ and $\alpha$.
\end{theorem}

\begin{proof}
We use notation $\mathcal{K}(t)$ to denote $(1+\alpha {K}_{\Delta T,q}((1+\alpha)t))$ to simplify the argument.
First, for fixed $\xi$ and $\tilde{\xi}$, $\eta$ and $\tilde{\eta}$ have the same invariant measure. If we denote the averaged solutions to (\ref{eq:model1}) and (\ref{eq:model3}) by $\Xi$ and $\tilde{\Xi}$ respectively, they satisfy
\begin{equation}\label{eq:avg1}\frac{d}{dt}{\Xi}=\bar{f}(\Xi),\qquad 0\leq t\leq \Delta T\end{equation}
and
\begin{equation}\label{eq:avg2}\frac{d}{dt}{\tilde{\Xi}}=\mathcal{K}(t)\bar{f}(\tilde{\Xi}),\qquad 0\leq t\leq \frac{\Delta T}{1+\alpha}.\end{equation}
$\mathcal{K}(t)$ satisfies
\begin{equation}\label{eq:kerprop3}
\begin{split}
\int_{0}^{\frac{\Delta T}{1+\alpha}}\mathcal{K}(t)dt=&\int_0^{\frac{\Delta T}{1+\alpha}}\left(1+\alpha{K}_{\Delta T,q}((1+\alpha)t)\right)dt\\
=&\frac{\Delta T}{1+\alpha}+\frac{\alpha}{1+\alpha}\int_0^{\Delta T} {K}_{\Delta T,q}(s)ds\qquad \textrm{with } s=(1+\alpha)t\\
=&\Delta T\quad
\end{split}
\end{equation}
For (\ref{eq:avg2}), by using (\ref{eq:kerprop3}) and change of time $\tau$ such that $d\tau/dt={K}(t)$, we can verify that
$$\Xi(\Delta T)=\tilde{\Xi}(\frac{\Delta T}{1+\alpha}).$$

Now we compare $\tilde{\Xi}(\frac{\Delta T}{1+\alpha})$ and $\tilde{\xi}(\frac{\Delta T}{1+\alpha})$. First,
$$\tilde{\xi}\left(\frac{\Delta T}{1+\alpha}\right)=\int_{0}^{\frac{\Delta T}{1+\alpha}}\mathcal{K}(t)f_0(\tilde{\xi},\tilde{\eta} )dt.$$
Let $g_0(t,t/\epsilon):=f_0(\tilde{\xi}(t),\tilde{\eta}(t))$ and $\bar{f}(t)=\int g_0(t,s)ds$ where $g_0$ is 1-periodic in the second variable. Then
$$\tilde{\Xi}\left(\frac{\Delta T}{1+\alpha}\right)=\int_{0}^{\frac{\Delta T}{1+\alpha}}\mathcal{K}(t)\bar{f}(t)dt.$$

Therefore we have
$$\tilde{\xi}\left(\frac{\Delta T}{1+\alpha}\right)-\tilde{\Xi}\left(\frac{\Delta T}{1+\alpha}\right)=\int_{0}^{\frac{\Delta T}{1+\alpha}}\mathcal{K}(t)g_1(t,t/\epsilon)dt$$
where 
$$g_1(t,t/\epsilon)=g_0(t,t/\epsilon)-\bar{f}(t).$$
Using Lemma \ref{lem:thmosc1}, we prove the theorem for the case when $g_0(t,t/\epsilon)$ is periodic in the second variable.
\end{proof}

\begin{lemma}\label{lem:thmosc1}
$$\left\|\int_0^{\frac{\Delta T}{1+\alpha}}\mathcal{K}(t)g_1(t,t/\epsilon)dt\right\|_{\infty}\leq
C_1\ep
+C_2\sum_{r=1}^q\frac{\alpha^r\epsilon^{r+1}}{{\Delta T}^{r-1}}
+C_3\frac{(\alpha\epsilon)^{q+1}}{{\Delta T}^q}.$$
where $C_i,i=1,2,3$ are constants independent of $\epsilon,\alpha$ and $\Delta T$.
\end{lemma}

\begin{proof}
Partition the interval, $(0,\frac{\Delta T}{1+\alpha})$, into $N$ uniform subintervals, $(t_n,t_{n+1}), n=0,1,...,N-1$ such that
$$t_0=0,\quad t_N=\frac{\Delta T}{1+\alpha}$$
and
$$|t_{n+1}-t_n|=\epsilon.$$
This condition requires that $N=\frac{\Delta T}{(1+\alpha)\epsilon}$.

For $t_n\leq t\leq t_{n+1}, n=0,1,...,N-1$, using Taylor series expansion of $g_1(\cdot,\cdot)$ in the first variable at $t=t_{n+1/2}=\frac{t_{n+1}+t_{n}}{2}$, we have
$$g_1(t,t/\epsilon)=g_1(t_{n+1/2},t/\epsilon)+\partial_1 g_1(t_{n+1/2},t/\epsilon)(t-t_{n+1/2})+\mathcal{O}(\epsilon^2).$$

$\|\partial_1g\|_{\infty}$ is bounded and independent of $\alpha$ and $\epsilon$. The second term gives $\ep$ order term and after integration on $[t_n,t_{n+1}]$, it becomes $\ep^2$ order. There are $N=\frac{\Delta T}{(1+\alpha)\ep}$ intervals, therefore, we have
$$\int_0^{\frac{\Delta T}{1+\alpha}}\mathcal{K}(t)g_1(t,t/\epsilon)dt=\sum_{n=0}^{N-1}\left[\int_{t_n}^{t_{n+1}}\mathcal{K}(t)g_1(t_{n+1/2},t/\epsilon)dt\right]+\mathcal{O}(\Delta T\epsilon)$$

We further analyze the first term on the right hand side using integration by parts, which gives
\begin{equation}\label{eq:lem2pf1}
\sum_{n=0}^{N-1}\left\{ \epsilon\mathcal{K}(t)g^{[1]}(t_{n+1/2},t/\epsilon)\Big|_{t_n}^{t_{n+1}}
-\int_{t_n}^{t_{n+1}}\epsilon \mathcal{K}'(t)g^{[1]}(t_{n+1/2},t/\epsilon)dt\right\}
\end{equation}
where $g^{[1]}(t_{n+1/2},s)$ is an antiderivative of $g_1(t_{n+1/2},s)$ such that $\int g^{[1]}(t_{n+1/2},s)ds=0$. From this mean zero condition $g^{[1]}(t_{n+1/2},s)$ is periodic in $s$.

For the first term of (\ref{eq:lem2pf1}), after rearrangement of the summation, 
\begin{equation}\label{eq:lem2pf2}
\begin{split}
\sum_{n=0}^{N-1} \epsilon\mathcal{K}(t)g^{[1]}(t_{n+1/2},t/\epsilon)\Big|_{t_n}^{t_{n+1}} =\epsilon&\times\left\{\mathcal{K}(t_N)g^{[1]}(t_{N-1/2},\frac{t_{N}}{\epsilon})\right.\\
&-\sum^{N-1}_{n=1} \mathcal{K}(t_n)\left(g^{[1]}(t_{n+1/2},\frac{t_n}{\epsilon})-g^{[1]}(t_{n-1/2},\frac{t_n}{\epsilon})\right)\\
&\left.-\mathcal{K}(t_0)g^{[1]}(t_{1/2},\frac{t_0}{\epsilon})\right\}
\end{split}
\end{equation}

Using Lemma \ref{lem:thmosc2} with the facts that $\|\mathcal{K}\|_{\infty}=\mathcal{O}(\alpha)$ and $N=\mathcal{O}(\frac{\Delta T}{\alpha\epsilon})$, we estimate the second term of (\ref{eq:lem2pf2}),
$$\epsilon\sum_{n=1}^{N-1} \mathcal{K}(t_n)\left(g^{[1]}(t_{n+1/2},\frac{t_n}{\epsilon})-g^{[1]}(t_{n-1/2},\frac{t_n}{\epsilon})\right)=\mathcal{O}(\Delta T\epsilon).$$
For $t=t_0$ and $t_N$, $\mathcal{K}(t)=1$, and we have $\mathcal{O}(\epsilon)$ for the first and last terms of (\ref{eq:lem2pf2}).

For the second term of (\ref{eq:lem2pf1}), we do the same procedure except that now we have
$$\|\mathcal{K}'(t)\|_{\infty}=\mathcal{O}(\frac{\alpha^2}{\Delta T})$$
$$\mathcal{K'}(t)=0\quad \textrm{for }t=t_0,t_N,$$
and
$$g^{[1]}(t_{n+1/2},s)-g^{[1]}(t_{n-1/2},s)=\mathcal{O}(\epsilon).$$

After integration by parts, we have
$$\sum_{n=0}^{N-1}\left\{-\int_{t_n}^{t_{n+1}}\epsilon \mathcal{K}'(t)g^{[1]}(t_{n+1/2},t/\epsilon)dt\right\}=\sum_{n=0}^{N-1}\int_{t_n}^{t_{n+1}} \epsilon^2\mathcal{K}''(t)g^{[2]}(t_{n+1/2},t/\epsilon)dt+\mathcal{O}(\alpha \epsilon^2)$$
where $g^{[2]}(t_{n+1},s)$ is an antiderivative of $g^{[1]}(t_{n+1},s)$ such that $\int g^{[2]}(t_{n+1/2},s)ds=0$.
Now $L_{\infty}$ estimate of the right hand side is
$$\mathcal{O}(N\times \epsilon \times \epsilon^2\times \frac{\alpha^3}{\Delta T^2})=\mathcal{O}(\frac{\alpha^2\epsilon^2}{\Delta T})$$
If $d^r \mathcal{K}(t)/dt^r=0$ for $t=t_0,t_N$, and $r\leq q$, we can repeat the same procedure and this proves the lemma.
\end{proof}

\begin{lemma}\label{lem:thmosc2}
$$g^{[1]}(t_{n+1/2},s)-g^{[1]}(t_{n-1/2},s)=\mathcal{O}(\epsilon)$$
\end{lemma}
\begin{proof}
For $t_{n-1/2}\leq t_*\leq t_{n+1/2}$,
$$g^{[1]}(t_*,s)=\int_0^s g_1(t_*,\tau)d\tau-C(t_*)$$
where constant $C(t_*)$ is given by
$$C(t_*)=\int_0^1\int_0^s g_1(t_*,\tau)d\tau ds.$$

Now
\begin{equation}
\begin{split}
\nonumber g^{[1]}(t_{n+1/2},s)-g^{[1]}(t_{n-1/2},s)=&\int_0^{s}\left(g_1(t_{n+1/2},\tau)-g_1(t_{n-1/2},\tau)\right)d\tau\\
&+C(t_{n+1/2})-C(t_{n-1/2})
\end{split}
\end{equation}

Using 
$$g_1(t_{n+1/2},\tau)=g_1(t_{n-1/2},\tau)+\mathcal{O}(\epsilon),$$
we have 
$$\int_0^{s}\left(g_1(t_{n+1/2},\tau)-g_1(t_{n-1/2},\tau)\right)d\tau=\mathcal{O}(\epsilon)$$
Also for the constant terms, 
\begin{equation}
\begin{split}
\nonumber C(t_{n+1/2})-C(t_{n-1/2})=&\int_0^1\int_0^{s}\left(g_1(t_{n+1/2},\tau)-g_1(t_{n-1/2},\tau)\right)d\tau ds\\
 =&\mathcal{O}(\epsilon)
\end{split}
\end{equation}
This proves the Lemma.
\end{proof}

\section{Numerical Experiments}\label{NUME}
In this section, we apply VSHMM to dissipative and highly oscillatory problems. The result of VSHMM is compared with MSHMM and FLAVORS. The convergence of the second order scheme is also verified for (\ref{eq:model2}) which does not have a priori identification of slow and fast variables. As the last test problem, VSHMM is applied to a stellar orbit problem \cite{stellar,stellar1,stellar2}. The comparison of three methods, MSHMM, FLAVORS, and VSHMM shows that VSHMM captures slow variables with higher accuracy than MSHMM and FLAVORS. 

In choosing $K$ for the variable mesoscopic time step, the $L_{\infty}$ norm of $K$ plays more important role than the regularity conditions at the boundary. The error depending on the regularity of $K$ has the form of powers of $(\alpha\epsilon)$ and the restriction (\ref{eq:alphacond}) for choosing $\alpha$ shows that if $q\geq 1$, then this error is small. On the other hand, if $\|K\|_{\infty}$ is too large, then instantaneous value of $h(t)+\delta t$ at some time violates the condition (\ref{eq:alphacond}) and this violation deteriorate the convergence of VSHMM. In all numerical examples, we use $K(t)=\left(1+\cos({2\pi (t-1/2)})\right)$ which has $\|K\|_{\infty}=2$ and $q=1$. Also, the micro steps for all methods are same. Therefore, the computational efficiency between DNS, MSHMM, FLAVORS, and VSHMM follows (\ref{eq:efficiency}).

\subsection{Stiff Dissipative Case}
We begin with a dissipative example to show the effect of resolving the transient behavior accurately.

The following problem has explicit form of slow and fast variables,
	\begin{equation}\label{eq:dis}
	\begin{split}
	\frac{d}{dt}{\xi}=&1+\frac{\xi+\eta}{2},\qquad 0< t< 1\\
	\frac{d}{dt}{\eta}=&\frac{\xi-\eta}{\epsilon}
	\end{split}
	\end{equation}
	with initial value $(x(0),y(0))=(-1,1)$.
	
We can verify that the averaged equation is given by
	\begin{equation}
	\frac{d\Xi}{dt}=1+\Xi, \qquad\Xi(0)=-1
	\end{equation}

	\begin{figure}
		\begin{center}
		\includegraphics[width=6cm]{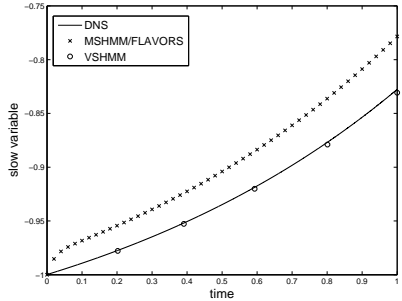}
		\caption{Dissipative case (\ref{eq:dis}). Plot of the slow variable $\xi$ of (\ref{eq:dis}) computed by DNS, MSHMM, FLAVORS and VSHMM. $\epsilon=2\times 10^{-4}$ and $\alpha=100$ for VSHMM, MSHMM and FLAVORS.}\label{fig:dissipative}
		\end{center}
	\end{figure}
	
	Figure \ref{fig:dissipative} shows the numerical result from the various methods for  $\epsilon=2\times 10^{-4}$ and $\alpha=100$. For VSHMM, MSHMM and FLAVORS, a fourth order Runge-Kutta method is used for full scale integrator while a second order Runge-Kutta method is used for mesoscopic integration. In this case, MSHMM and FLAVORS produce errors greater than the error from VSHMM because of the error from inappropriately resolved transient behavior at the beginning.
VSHMM approximates the slow variables more accurately than MSHMM and VSHMM at coarse time points with a time step $\Delta T=0.2$.

\subsection{Highly Oscillatory Case}
The next three numerical experiments are highly oscillatory problems where the eigenvalues of the Jacobian of $f_1$ are all imaginary. 
The first two problems are expanding spiral problems. In the first case, the fast variable has a fixed angular period while in the second problem, the fast variable has variable periods depending on the slow variable and fully nonlinear. The last problem is a well studied system taken from the theory of stellar orbits in a galaxy \cite{stellar,stellar1,stellar2}.

\subsubsection{Constant Angular Period Case}
The first oscillatory case is an expanding spiral problem for complex $x$ with constant angular period. The equation is given by
	\begin{equation}\label{eq:const}
		\frac{dx}{dt}=\frac{x}{4}+\frac{5Real(x)x}{|x|}+\frac{ix}{\epsilon},\qquad x\in\mathbb{C}
	\end{equation}
	$$x(0)=1.$$
The slow and fast variables are $\xi=|x|$ and $\eta=\arg(x)$ respectively and it can be verified that
\begin{equation}\label{eq:op}
\begin{split}
\frac{d\xi}{dt}=&\frac{\xi}{4}+5\xi\cos(\eta)\\
\frac{d\eta}{dt}=&\frac{1}{\epsilon}\\
(\xi(0),\eta(0))=&(1,0)
\end{split}
\end{equation}
which has constant periodic force for $\xi$.

	\begin{figure}
		\begin{center}
		\subfloat[global]{\includegraphics[width=7cm]{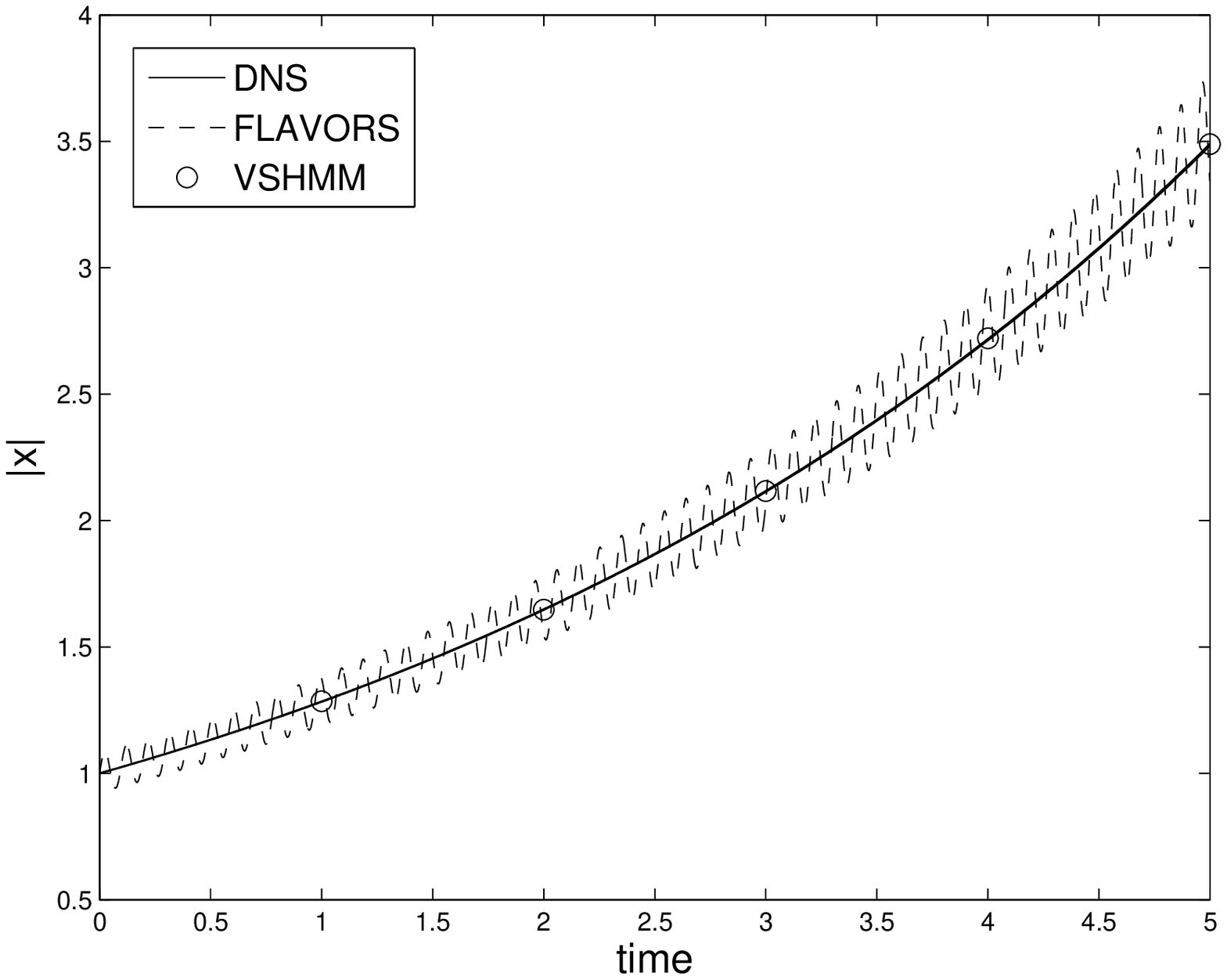}}
		\subfloat[local]{\includegraphics[width=7cm]{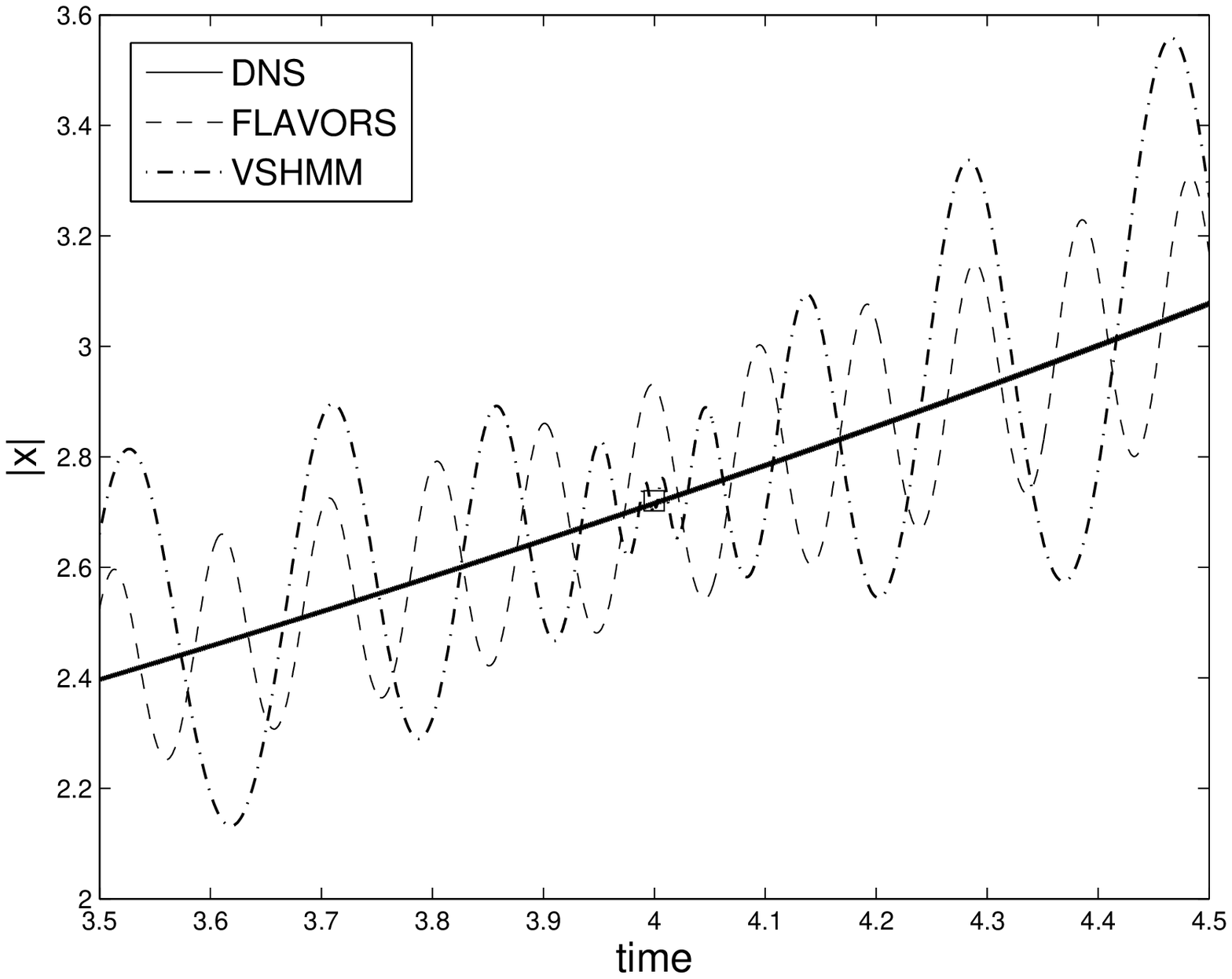}\label{fig:local}}
		\caption{Constant angular period case (\ref{eq:const}). Plot of the slow variable $\xi=|x|$ of (\ref{eq:op}) computed by DNS, FLAVORS and VSHMM. Global (left) and local (right) representations of each solution in [0,3]. $\epsilon=\frac{1}{3400}$ and $\alpha=50$ for VSHMM and FLAVORS.}\label{fig:op}
		\end{center}
	\end{figure}
Figure \ref{fig:op} shows the slow variable over time with a locally magnified plot in the neighborhood of $t=4$. As shown in Figure \ref{fig:local}, the proposed method, like FLAVORS, has amplified oscillations except in the neighborhood of the specified macro time interval. Once it reaches the neighborhood of the macro time interval, the new method has less oscillations and converges faster to the averaged solution.

Next we verify the order of accuracy for the second order mesoscopic integration methods. In VSHMM, there are several error terms from various factors - order of accuracy of each integrator and regularity conditions and the $L_{\infty}$ norm of $K$ for variable stepping. To see the second order behavior, other error terms must be well controlled to be much smaller than the error of the second order integration. For this purpose, we choose $\alpha=52.67$ which was obtained from numerical tests. Figure \ref{fig:order} shows the errors of the first and second order methods with the analytic effective solution. For sufficiently small average mesoscopic step size, $\alpha \delta t$, the first order method shows linear decrease of error as expected while the second order method shows quadratic decrease of error for relatively large mesoscopic step sizes. The error of the second order method becomes flat for much smaller mesoscopic step sizes because the error is dominated by the averaging error which is of order $\epsilon$.

	\begin{figure}
		\begin{center}
		\subfloat[]{\includegraphics[width=7cm]{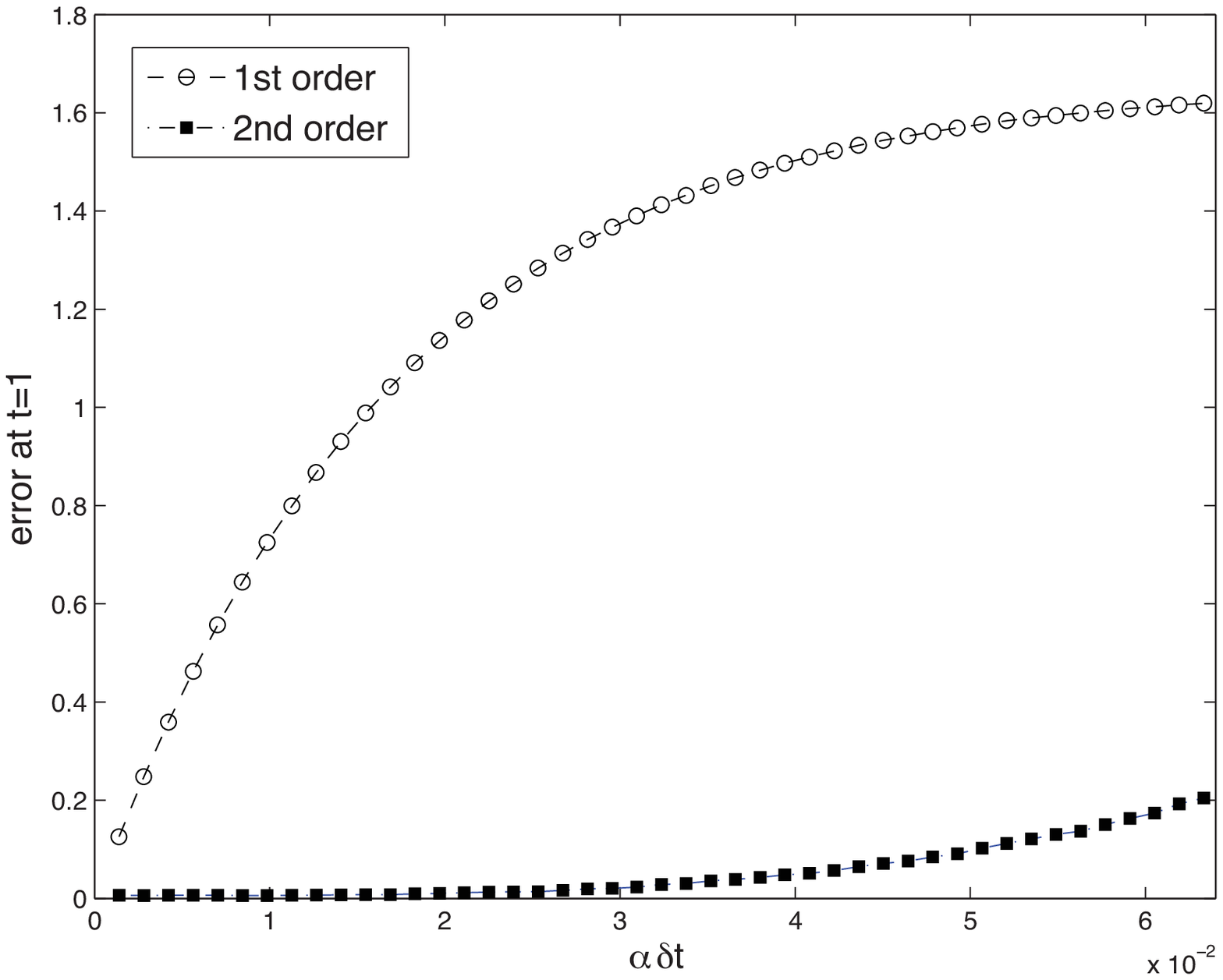}}
		\subfloat[]{\includegraphics[width=7cm]{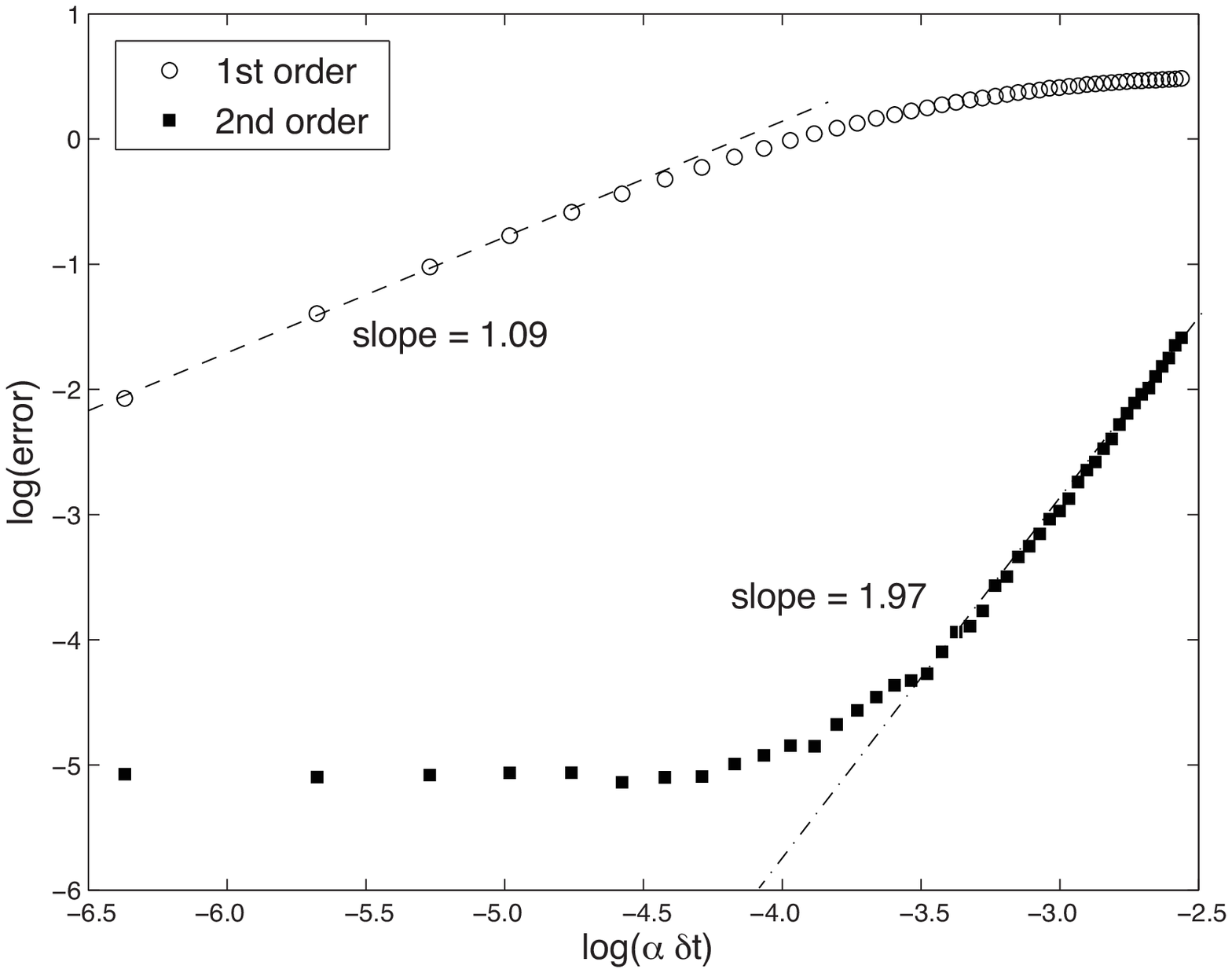}}
		\caption{Errors of the first and second order VSHMM for (\ref{eq:const}). $\epsilon=10^{-4}$ and $\alpha=52.67$.}\label{fig:order}
		\end{center}
	\end{figure}

\subsubsection{Fully Nonlinear Case}
The second test problem is also an expanding spiral problem in $\mathbb{C}^2$ where two spirals are coupled. This system is more general than the previous oscillatory problem in that it is not angular periodic with nonlinear $f_1$. For a fixed $\xi$, the periodicity of each component of $\eta$ depends on $\xi$ and the oscillation of $\xi$ comes from $\eta_1$ and $\eta_2$ simultaneously which has an irrational initial ratio.

\begin{equation}\label{eq:nonlinear}
\begin{split}
	\frac{dx}{dt}=&\frac{x}{\sqrt{|x|}^{3}}+5\left(\frac{Real(x)}{|x|}+\frac{Real(y)}{|y|}\right)\frac{x}{|x|}+\frac{i|x|x}{\epsilon},\\
	\frac{dy}{dt}=&\frac{y}{\sqrt{|y|}^{3}}+\frac{Real(y)y}{|y|^2}+\frac{i|y|y}{\sqrt{2}\epsilon},\\
	&(x(0),y(0))=(1,i)
\end{split}
\end{equation}
The slow and fast variables are given by $\xi=(|x|,|y|)$ and $\eta=(\arg_1(x),\arg_1(y))$ respectively.
\begin{figure}
	\begin{center}
	\subfloat[global]{\includegraphics[width=7cm]{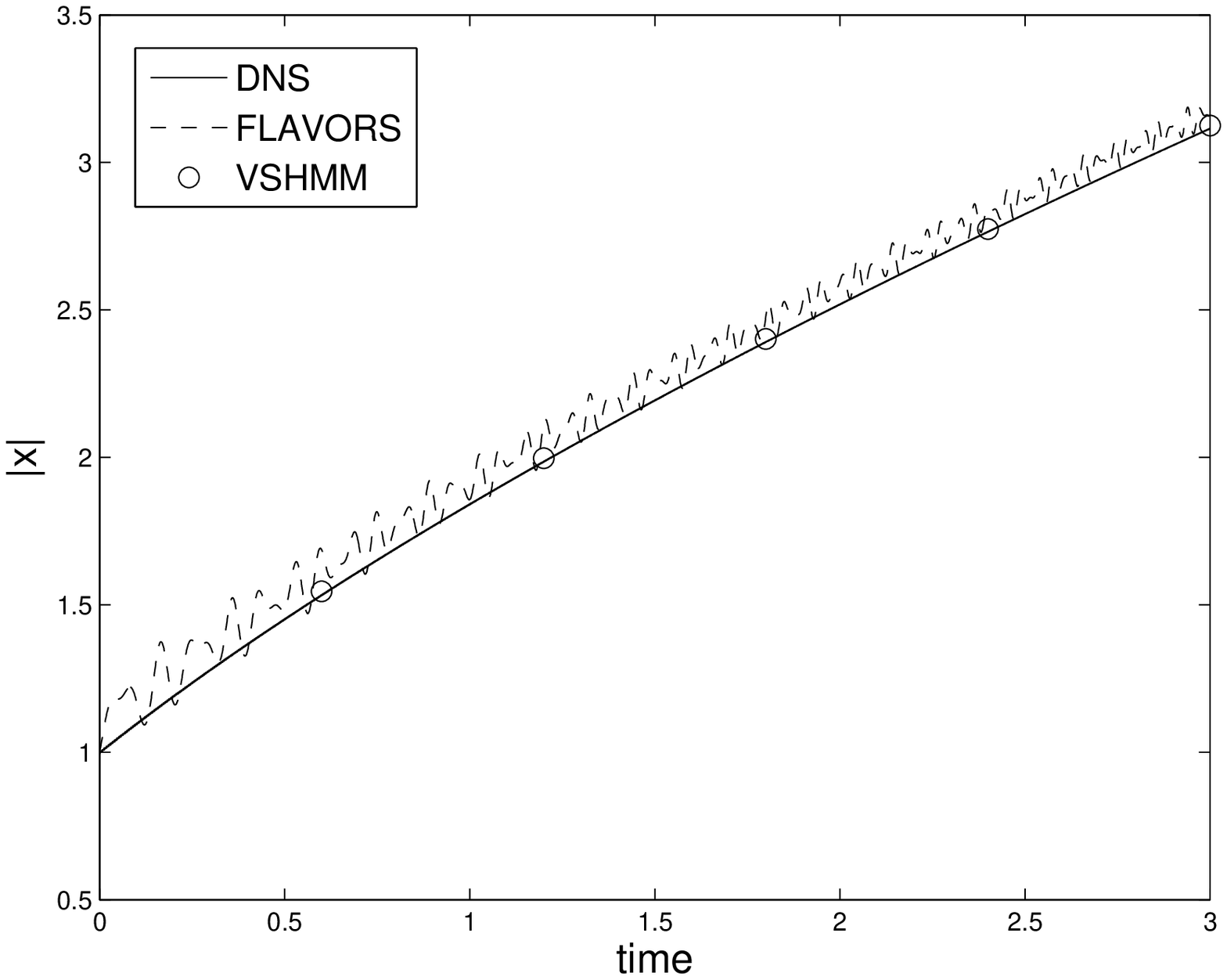}}
	\subfloat[local]{\includegraphics[width=7cm]{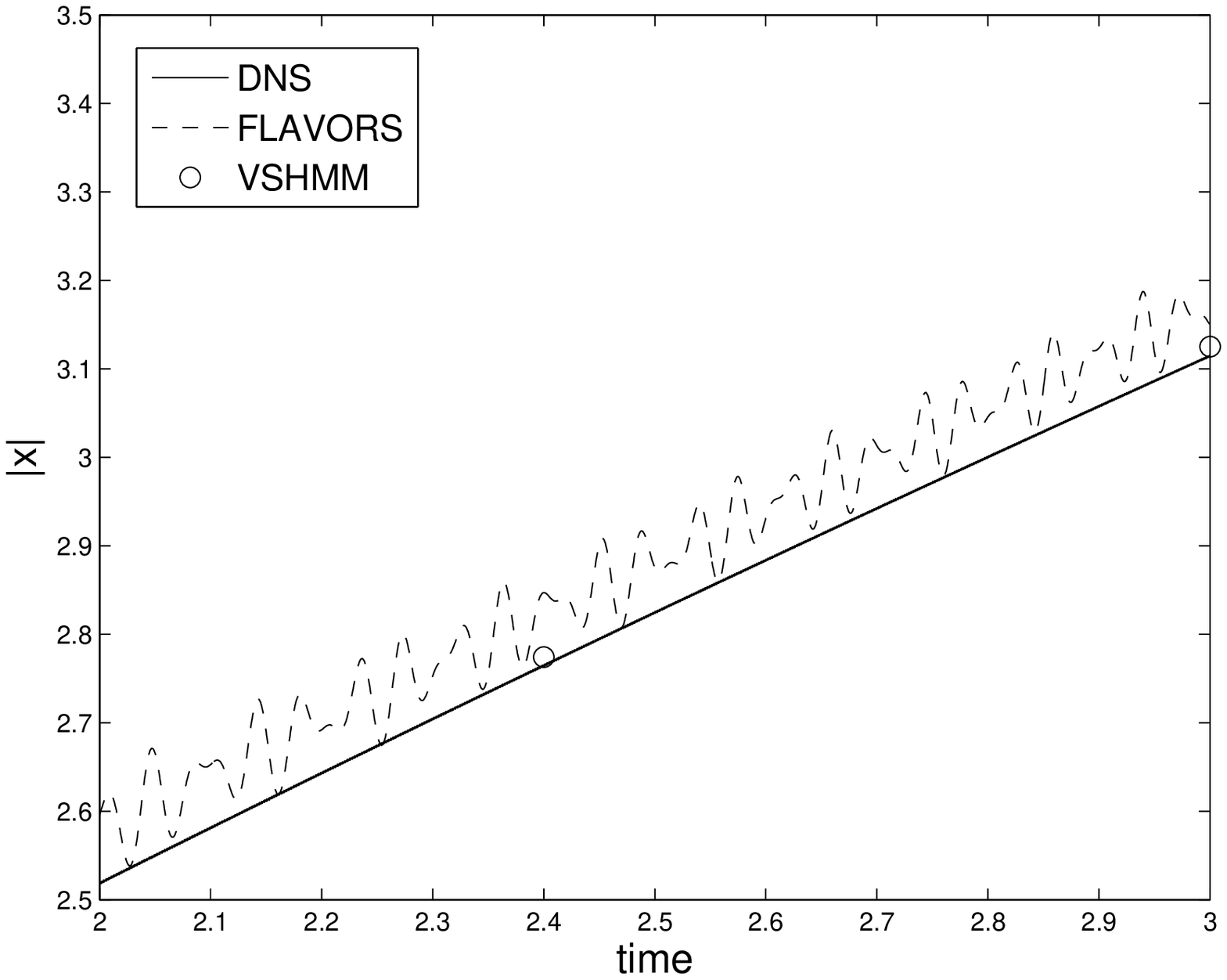}}
	\caption{Fully nonlinear case (\ref{eq:nonlinear}). Plot of the first slow variable $\xi_1=|x|$ computed by DNS, FLAVORS and VSHMM. $\epsilon=5\times 10^{-4}$ and $\alpha=50$ for VSHMM and FLAVORS.}\label{fig:nonlinear}
	\end{center}
\end{figure}
Figure \ref{fig:nonlinear} shows the result of VSHMM and FLAVORS for (\ref{eq:nonlinear}) with $\epsilon=5\times 10^{-4}$ and $\alpha=50$. A fourth order Runge-Kutta method is used for micro step simulation and a second order Runge-Kutta method for mesoscopic time step simulation. VSHMM captures the correct slow variable on the macro time interval, $\Delta T=0.6$. In Figure \ref{fig:op}, it is clear that averaging the FLAVORS solution improves the approximation of the effective solution. Figure \ref{fig:nonlinear} shows that this is not always the case.

\subsubsection{Stellar Orbit Problem with Resonance}
The last numerical experiment is a well studied system taken from the theory of stellar orbits in a galaxy \cite{stellar,stellar1,stellar2},
\begin{equation}
\begin{split}
\nonumber r_1^{''}+a^2r_1=&\epsilon r_2^2,\\
r_2^{''}+b^2r_2=&2\epsilon r_1 r_2
\end{split}
\end{equation}
where $r_1$ is the radial displacement of the orbit of a star from a reference circular orbit, and $r_2$ is the deviation of the orbit from the galactic plane. Here $t$ is actually the angle of the planets in a reference coordinate system. Using an appropriate change of variables, the system can be written in the following form \cite{stellar}
\begin{equation}\label{eq:stellar}
\frac{d\mathbf{x}}{dt}=\frac{1}{\epsilon}\begin{pmatrix}0&a&0&0\\-a&0&0&0\\0&0&0&b\\0&0&-b&0\end{pmatrix}\mathbf{x}+\begin{pmatrix}0\\x_3^2/a\\0\\2x_1x_2/b\end{pmatrix},\quad\mathbf{x}(0)=\begin{pmatrix}1\\0\\1\\0\end{pmatrix},\quad \mathbf{x}\in\mathbb{R}^4
\end{equation}
	with $a=2$, and $b=1$.
In \cite{stellar, IHMM}, it is verified that in the case of $a=\pm 2b$, the system is in resonance and has three hidden slow variables $\xi_i:\mathbb{R}^4\to\mathbb{R}, i=1,2,3$, are given by
	\begin{equation}\label{eq:slowstellar}
	\begin{split}
	\xi_1=&x_1^2+x_2^2\\
	\xi_2=&x_3^2+x_4^2\\
	\xi_3=&x_1x_3^2+2x_2x_3x_4-x_1x_4^2
	\end{split}
	\end{equation}
The resonance of oscillatory modes generates lower order effects, that are captured by VSHMM.
	\begin{figure}
		\begin{center}
		\includegraphics[width=9cm]{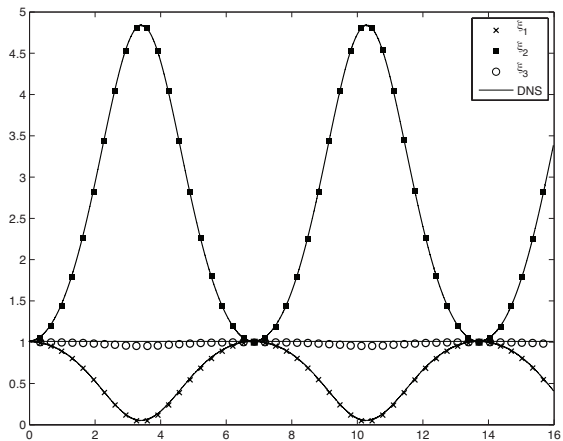}
		\caption{Stellar orbit problem (\ref{eq:stellar}). Plot of the three slow variables (\ref{eq:slowstellar}) computed by DNS (real line) and VSHMM (marked with cross, square and circle). $\epsilon=10^{-4}$ and $\alpha=100$.\label{fig:stellar}}
		\end{center}
	\end{figure}
In Figure  \ref{fig:stellar}, we present a numerical result of our method for $\epsilon=10^{-4}$ and $\alpha=100$.

\section{Conclusions}
We have presented a new multiscale integrator VSHMM for stiff and highly oscillatory dynamical systems. It controls the transient and the amplified oscillations of MSHMM and FLAVORS while preserving the computational complexity and general structure of these methods.This results in an overall higher accuracy. The main idea of the error control is to use variable mesoscopic step sizes determined by special functions satisfying moment and regularity conditions. The proposed method is restricted to ordinary differential equations with two scales. Applications to stochastic differential equations and problems with more than two scales will be reported in a forthcoming paper by the authors \cite{VSSHMM}.

\section*{Acknowledgments}
The research was partially supported by NSF grants DMS-1027952 and DMS-1217203. The authors thank Christina Frederick for comments and a careful reading of the manuscript


\begin{thebibliography}{99} 
\bibitem{generalHMM}
	\newblock A. Abdulle, W. E, B. Engquist, and E. Vanden-Eijnden,
	\newblock \emph{The Heterogeneous Multiscale Method},
	\newblock Acta Numerica, \textbf{21}, (2012), 1--87.

\bibitem{IHMM}
	\newblock G. Ariel, B. Engquist, S. Kim, Y. Lee, and R. Tsai,
	\newblock \emph{A multiscale method for highly oscillatory dynamical systems using a Poincar{\'e} map  type technique},
	\newblock Journal of Scientific Computing, (2012), 10.1007/s10915-012-9656-x.

\bibitem{stellar}
	\newblock G. Ariel, B. Engquist, and R. Tsai,
	\newblock \emph{A multiscale method for highly oscillatory ordinary differential equations with resonance},
	\newblock Mathematics of Computation, \textbf{78}, (2008), 929--956.

\bibitem{threescale} 
	\newblock G. Ariel, B. Engquist, and R. Tsai,
	\newblock \emph{Oscillatory systems with three separated time scales Ð analysis and computation},
	\newblock Lecture Notes in Computational Science and Engineering,\textbf{82}, (2011), Springer-Verlag.

\bibitem{ISERLES}
	\newblock M. Condon, A. Dea\~no, and A. Iserles,
	\newblock \emph{On second-order differential equations with highly oscillatory forcing terms},
	\newblock Proc. R. Soc. Lond. Ser. A Math. Eng. Sci., \textbf{466}, 2010, 1809--1828.
	
\bibitem{MSHMM}
	\newblock W. E, W. Ren, and E. Vanden-Eijnden,
	\newblock \emph{A general strategy for designing seamless multiscale methods},
	\newblock Journal of Computational Physics, \textbf{228}, (2009), 5437--5453.

\bibitem{HMM}
	\newblock B. Engquist and Y. Tsai,
	\newblock \emph{Heterogeneous multiscale methods for stiff ordinary differential equations},
	\newblock Mathematics of Computation, \textbf{74}, (2005), 1707--1742.

\bibitem{FATKULLIN}
	\newblock I.Fatkullin and E. Vanden-Eijnden,
	\newblock \emph{A computational strategy for multi scale systems with applications to Lorenz 96 model},
	\newblock J. Comput. Phys. \textbf{200}, (2004), 606--638.
	
\bibitem{ODE2}
	\newblock E. Hairer and G. Wanner,
	\newblock \emph{Solving ordinary differential equations II},
	\newblock Springer Series in Computational Mathematics, \textbf{14}, (1996), Springer-Verlag.

\bibitem{ODE3}
	\newblock E. Hairer, C. Lubich, and G. Wanner,
	\newblock \emph{Geometric numerical integration},
	\newblock Springer Series in Computational Mathematics, \textbf{31}, (2010), Springer-Verlag.

\bibitem{EXPINT1}
	\newblock M. Hochbruck, C. Lubich, and H. Selhofer,
	\newblock \emph{Exponential integrators for large systems of differential equations},
	\newblock SIAM Journal on Scientific Computing, \textbf{19}, (1998), 1552--1574.

\bibitem{GAUTSCHI}
	\newblock M. Hochbruck, and C. Lubich,
	\newblock \emph{A Gautschi-type method for oscillatory second-order differential equations},
	\newblock  Numer. Math., \textbf{83}, (1999), 403--426.
	
\bibitem{stellar1}
	\newblock J. Kevorkian and J. D. Cole,
	\newblock \emph{Perturbation methods in applied mathematics},
	\newblock Applied Mathematical Sciences, \textbf{34}, (1980), Springer-Verlag.

\bibitem{stellar2}
	\newblock J. Kevorkian and J.D. Cole,
	\newblock \emph{Multiple scale and singular perturbation methods},
	\newblock Applied Mathematical Sciences, \textbf{114}, (1996), Springer-Verlag.

\bibitem{TURBULENCE}
	\newblock Y. Lee and B. Engquist,
	\newblock \emph{Seamless multiscale methods for diffusion in incompressible flow},
	\newblock in preparation.
	
\bibitem{VSSHMM}
	\newblock Y. Lee and B. Engquist,
	\newblock \emph{Fast integrators for several well-separated scales},
	\newblock in preparation.
	
\bibitem{SUPERPARAMETRIZATION}
	\newblock A. J. Majda and M. J. Grote,
	\newblock \emph{Mathematical test models for superparametrization in anisotropic turbulence},
	\newblock Procedings of the National Academy of Science of the United States of America, \textbf{106}, (2009), 5470--5474

\bibitem{STUART}
	\newblock G. Pavliotis and A. Stuart,
	\newblock \emph{Multiscale methods : averaging and homogenization},
	\newblock Texts in Applied Mathematics, \textbf{53}, (2008), Springer-Verlag.

\bibitem{AVG1}
	\newblock J. A. Sanders and F. Verhulst,
	\newblock \emph{Averaging methods in nonlinear dynamical systems},
	\newblock Applied Mathematical Sciences, \textbf{59}, (1985), Springer-Verlag.
	
\bibitem{AVG2}
	\newblock J. M. Sanz-Serna and M. P. Calvo,
	\newblock \emph{Numerical hamiltonian problems},
	\newblock Applied Mathematical Sciences, \textbf{7}, (1994), Springer-Verlag.

\bibitem{FLAVORS}
	\newblock M. Tao, H. Owhadi and J. Marsden,
	\newblock \emph{Nonintrusive and structure preserving multiscale integration of stiff ODEs, SDEs, and Hamiltonian Systems with Hidden Slow Dynamics via Flow Averaging},
	\newblock Multiscale Modeling and Simulation, \textbf{8}, (2010), 1269--1324.
	
\bibitem{ERIC}
	\newblock E. Vanden-Eijnden,
	\newblock \emph{On HMM-like integrators and projective integration methods for systems with multiple time scales},
	\newblock Comm. Math. Sci., \textbf{5}, 495--505.
\end{thebibliography}
\end{document}